\documentclass[12pt, english]{article}
\usepackage{babel}
\usepackage[latin1]{inputenc}
\usepackage{amsmath}
\usepackage{amssymb}
\usepackage{amsthm}

\theoremstyle{plain}
\newtheorem{Thm}{Theorem}[section]

\newtheorem{Conjecture}[Thm]{Conjecture}

\theoremstyle{definition}

\newtheorem{Def*}{Definition}

\newcommand{\nim}[1]{\left<#1\right>}

\begin{document}

\title {The strange algebra of combinatorial games}

\author{Johan W\"astlund\\
\small Department of Mathematics\\[-0.8ex]
\small Chalmers University of Technology,\\[-0.8ex] \small S-412 96 Gothenburg, Sweden\\[-0.8ex]
\small \texttt{wastlund@chalmers.se}
}
\date{\small \today}

\maketitle

\begin{abstract} We present an algebraic framework for the analysis of combinatorial games. This framework embraces the classical theory of partizan games as well as a number of mis\`ere games, comply-constrain games, and card games that have been studied more recently. It focuses on the construction of the quotient monoid of a game, an idea that has been successively applied to several classes of games.
\end{abstract}

\section{Introduction} \label{S:intro}
We describe an algebraic framework for the analysis of combinatorial games. Roughly speaking we generalize the classical theory of \cite{WW, ONAG}, in which games are regarded as elements in abelian groups, to a wider class of games to which we associate abelian monoids. 

The idea is not new. Apart from the theory developed in \cite{WW, ONAG} for games under \emph{normal play}, this approach has been suggested independently in at least three different contexts:
\begin{itemize}
\item In \cite{Propp} Jim Propp studied three-player impartial games. 
\item Thane Plambeck \cite{P05, P} and Plambeck and Aaron Siegel \cite{PS08} have obtained a beautiful theory of mis\`ere impartial games, in some cases solving long-standing open problems. See also \cite{Siegel}.
\item In \cite{JW2}, the author analyzed certain two-person card games with a method stemming directly from \cite{ONAG}, although these games fail to be combinatorial games in the strict sense. 
\end{itemize}
A central idea in all three cases was the definition of \emph{equivalence} of game states, shifting the focus from direct attempts at describing optimal play to determining the algebraic structure of the resulting quotient monoid. 

The present paper is an attempt to describe a general methodology embracing these examples and to refute the wide-spread belief that the theory developed in \cite{WW, ONAG} applies only to games under normal play. At the same time we find that several results like the solution to Moore's nim \cite{Moore} and the results on comply-constrain games \cite{GS, HR}, which have been regarded as isolated curiosities, also fit into this framework.

\section{Games}
The well-known theory of combinatorial games as developed in \cite{WW, ONAG} starts from a precise definition of \emph{partizan} games and the subclass of \emph{impartial} games. Here we argue that the ideas behind this theory are applicable to a variety of other games. A problem is that at the current state of knowledge we have no definition of a class of games which isn't either too general and abstract to be of any use, or too narrow to include all our examples. Therefore the following should be taken as an informal description of the sort of game where we expect the algebraic approach to be useful. 

\begin{itemize}
\item Our games usually have \emph{players}. Often there are two of them (and they go by names like Left and Right, Blue and Red, or East and West).

\item The games have \emph{positions}. When playing the game, one of the positions has to be chosen as the starting position, but usually there is nothing special about that position.

\item The game proceeds by a sequence of \emph{moves}. A move is an action taken by a player, and the move transforms one position into another.

\item There are rules that specify an \emph{outcome} of the game. Sometimes one player wins and the other loses. Some games allow for ties or draws. Yet other games have scoring systems that award points to the players. 

\end{itemize}

In the examples in this paper, the players always have complete information about the game and there are no chance moves or simultaneous moves. It is not necessary to make such stipulations, but there are reasons why we shouldn't expect the present approach to lead to new insights about yatzy, prisoner's dilemma, or soccer.

The following two criteria are still informal, but point towards where we expect our approach to be useful.

\begin{itemize}
\item It is easy to define (but not necessarily to compute!) the outcome as a function of the current position. For two-player zero-sum games of perfect information this function is defined by \emph{optimal play}, but this is not the only possibility. In \cite{Propp} a position in a three-player game is classified as \emph{queer} if no player can force a win.

\item Positions split naturally into \emph{components}. There are plenty of examples of such games in \cite{WW}. In \cite{JW2}, card deals are split into single-suit subgames.
\end{itemize}

Our approach thus tends to focus on positions and their outcomes, rather than on the players and how they choose their moves. As we shall see, it is not always obvious from the rules of a game what we should mean by a position and an outcome. In the next three sections, we describe how, by defining these concepts appropriately, we can set up an algebraic framework in which a solution to the game can sometimes be obtained.

\section{Positions} \label{S:position} In a game such as chess or tic-tac-toe, it is natural to study the computational problem of finding, for a given position, its outcome under optimal play. To evaluate a position, we have to know not just how the pieces are located on the board, but also which of the two players is about to move. It is natural to think of each position as carrying with it a flag that shows whose turn it is. 

One of the great insights that led to the development of combinatorial game theory \cite{WW, ONAG} was that by separating the position from this flag, we can regard a ``sum'' of several games as a game in its own right. A player makes a move in the compound game by making a move in one of the components, leaving the rest of them unchanged. This requires each position to have a specified set of move options for each player, since the move order is not necessarily alternating within each component.

The games we study all have in common that their positions split naturally into components in one way or another. In each game we therefore regard the set of positions as an abelian monoid. For this to be possible we have to gather the ``global'' information about the game state (like which one of the players is about to move) into a \emph{flag}. To actually play the game, we would have to specify a position and a flag as the starting point. As we shall see, there may be other sorts of flags than the move-flag.

\section{The outcome of a game} \label{S:outcome} 
Our goal when analyzing a particular game is to predict its outcome from the starting position, assuming optimal play. However, we have to refine our concept of outcome. We say that a realization (possible line of play) of a game leads to a \emph{result}. A player wins, the game is drawn, or points are awarded etc. The way we have defined the concept of position, the result under optimal play will depend on the flag. Therefore when we speak of the \emph{outcome} of a position, we shall mean a function that associates a result to each value of the flag. 

In the classical theory, there are two values of the flag, ``Left to move'' and ``Right to move'', and there are two possible results of a game, ``Left wins'', and ``Right wins''. Consequently, positions are classified into four \emph{outcome classes}. These are \emph{Left wins} (no matter who starts), \emph{Right wins}, \emph{First player wins} and \emph{Second player wins}. These classes go by the names Positive, Negative, Fuzzy and Zero games respectively. If we would classify chess positions in the same manner, an outcome would be a function from the flags ``White to move'' and ``Black to move'' to the three potential results of a game, ``White wins'', ``Draw'', and ``Black wins''. There would therefore be $3^2=9$ outcome classes. 

\section{The algebraic framework} Let $A$ be the set of positions in a game. It is not necessarily clear from the rules of a game that there is a unique correct way to define $A$, but we assume that we have set things up so that $A$ is an abelian monoid. We use additive notation, since this is consistent with most of the literature, although Plambeck and Siegel \cite{P05, PS08} use multiplicative notation. We let $U$ be the set of outcomes. For the moment we completely disregard the players and the rules of the game, and just let $\chi:A\to U$ be the function that maps positions to outcomes. 

Perhaps we should say that $\chi$ maps positions to outcomes \emph{under optimal play}. However, this may be ambiguous, and lead to questions about how to define optimal play. In \cite{Propp}, Jim Propp discusses three-person games with a similar setup, mapping positions to outcome classes, without ever defining what is meant by optimal play! See also \cite{Loeb} for more on multi-player combinatorial games.

\subsection{Refinements} We say that a function $f:A\to B$ is a \emph{refinement} of $\chi$ if $\chi$ can be factored through $f$, that is, if there is a function $\phi:B\to U$ such that the diagram
\begin{equation}
\begin{matrix}
A & \overset{f}\longrightarrow & B\\
 & \chi\searrow & \downarrow \phi \\
 &  &  U
\end{matrix}
\end{equation}
commutes. A less fancy way of saying the same thing is that for all $x, y\in A$, we should have $\chi(x)=\chi(y)$ whenever $f(x) = f(y)$. The refinements that we discuss will sometimes be constructed artificially, and sometimes arise as outcome functions of other games (typically by a refinement of the preferences of the players). When no confusion can arise we identify a game with its outcome function. Therefore we can say that one game is a refinement of another.

\subsection{Homomorphisms} Normally there is no additive structure defined on the set $U$, and even if there is (see Section \ref{S:trick-taking}), the outcome function of a game is not in general a homomorphism. In order to use the additive structure of $A$, we make the following definition, which differs slightly from the algebra textbook in that we do not assume an additive structure to be given in advance on the image. A function $f:A\to B$ that maps $A$ to an arbitrary set $B$ is a \emph{homomorphism} if for $x, y\in A$, $f(x+y)$ is determined by $f(x)$ and $f(y)$. If this holds, then it is possible to define addition of elements in the image of $f$ so that for all $x, y\in A$, $f(x+y) = f(x) + f(y)$. Commutativity and associativity carry over from $A$ to the image of $f$, making it an abelian monoid. If $f$ is not surjective, then this means that $f$ can be a homomorphism even if this does not uniquely define addition on all of $B$, but this is not important.

\subsection{Homomorphic refinements} The idea of the algebraic method is that even if the outcome function $\chi$ of a game is not itself a homomorphism, it is sometimes not entirely incompatible with the addition on the set $A$ of positions. To exploit this, we try to construct a mapping $f$ from $A$ to a monoid $B$ which is at the same time a homomorphism and a refinement of the outcome function. If this can be done so that the function $\phi:B \to U$ can be described explicitly, then $B$ can serve as a stepping stone in order to describe the outcome function $A\to U$. In general it is reasonable to assume that the function $\phi$ will be easier to describe the more the elements of $A$ tend to map to the same elements of $B$. The following rather trivial theorem states that there is a natural candidate for the set $B$. However, this refinement is often easier to define than to describe explicitly.

\begin{Thm} There is a homomorphic refinement $\chi'$ of $\chi$ such that every homomorphic refinement of $\chi$ is a refinement of $\chi'$. 
\end{Thm} 

\begin{proof} We say that two positions $x, y \in A$ are \emph{equivalent} (with respect to $\chi$) if for every $z \in A$, $\chi(x+z) = \chi(y+z)$. If $f$ is a homomorphic refinement of $\chi$, and $x$ and $y$ are elements of $A$ such that $f(x)=f(y)$, then $x$ and $y$ have to be equivalent. Conversely, the natural mapping of elements of $A$ to equivalence classes is a homomorphic refinement of $\chi$.
\end{proof}

We let $A/\chi$ denote the quotient of $A$ with respect to equivalence. The strategy for solving combinatorial games that forms the theme of this paper is as follows:

\begin{enumerate}
\item Define \emph{positions}, how to \emph{add} them, and their \emph{outcomes} appropriately.

\item Find an explicit description of a homomorphic refinement $A\to B$, for instance the quotient map $A\to A/\chi$, and how $B$ maps to the set $U$ of outcomes. If necessary use guesswork and whatever ad hoc methods available.

\item Prove the correctness of the results stated in (2) by induction.
\end{enumerate} 

\section{Several examples} In this section, we demonstrate several examples of the algebraic approach. Some of these games are well-known, although their solutions have not been formulated in terms of monoid homomorphisms. In those cases we skip step (3).

\subsection{Nim} Nim is a two-person impartial game played with piles of counters. A move consists in removing any number of counters from a single pile. Players take turns moving, and a player unable to make a move loses. 

It is natural to represent a position as a formal sum of the piles. For example, a position with three piles of sizes 3, 4 and 6 could be written as $$ \nim{3} + \nim{4} + \nim{6}.$$ The symbols $\nim{1}, \nim{2}, \nim{3},\dots$ are regarded as formal variables, generating a free abelian monoid $H$ that we call the \emph{heap monoid}. 

There are two players and we can call them Left and Right, but since nim is an impartial game, if Left wins when Left starts, then by symmetry Right wins when Right starts, and vice versa. Therefore we only need to classify positions into $\mathcal{N}$-positions (Next player wins) and $\mathcal{P}$-positions (Previous player wins). But notice the difference between the \emph{results} (Left or Right wins) and the set of \emph{outcomes} $U = \{\mathcal{N}, \mathcal{P}\}$.

We identify nim with its outcome function $$\mathrm{nim}:H\to\{\mathcal{N}, \mathcal{P}\}.$$
According to the well-known solution found by C.~Bouton in 1902 \cite{Bouton}, when we pass to the quotient $H/\mathrm{nim}$, every position becomes equivalent to a single-pile position, which is found by so-called nim-addition (sometimes described as ``binary addition without carry'', or in computer language, ``XOR''). The quotient is therefore isomorphic to a direct sum of $\mathbf{Z}_2$'s, which we call the \emph{nim group}. $$H/\mathrm{nim} \cong \bigoplus_{i=0}^\infty{\mathbf{Z}_2}.$$

It has often been said that all impartial games are equivalent to nim. Implicit in this statement is not just the normal playing convention (last move wins), but also the assumption that games are added according to the classical theory, meaning that the move-order is alternating and a move is made by moving in exactly one of the components. As we shall see there are several impartial games which are, with respect to the natural additive structure, not isomorphic to nim. 
 
\subsection{Mis\`ere nim} The only difference between mis\`ere nim and ordinary nim is that in the mis\`ere form, the player who makes the last move \emph{loses}. The solution is well-known, but is usually presented as an ad hoc result which is not part of the general theory. Once the normal form of nim is solved, a solution to the mis\`ere form is obtained by observing that when only piles of size 1 remain, the game is decided by the parity of the number of piles. Therefore, whether the game is played by the normal or the mis\`ere playing convention, the player who leaves a single pile of size at least two, together with any number of piles of size 1, will lose. The other player removes the pile of size at least two, or leaves a single counter, according to parity and playing convention. 

Mis\`ere nim is played on the same set of positions as ordinary nim, and it is still an impartial game, so mis\`ere nim too is a function $H \to \{\mathcal{N}, \mathcal{P}\}$. To find the outcome of a position, we need to know two things, (i) the nim-sum of the pile sizes, and (ii) whether or not there is a pile of size at least 2 (a ``large'' pile). Mapping each position to the pieces of information (i) and (ii) is actually a homomorphism, since if we know (i) and (ii) for two positions $x$ and $y$, then we know (i) and (ii) for the sum $x+y$. Hence we can obtain an explicit homomorphic refinement of the mis\`ere nim function by mapping $H$ to the cartesian product of (i) the nim group, and (ii) a two-element monoid consisting of the element ``no large pile'' (which acts as zero element) and ``at least one large pile'', with the addition rule ``at least one large pile'' + ``at least one large pile'' = ``at least one large pile''. This monoid, which we denote by $B_1$, the boolean lattice of rank 1, is responsible for the fact that the quotient of $H$ with respect to mis\`ere nim is not a group. The quotient is actually isomorphic to the submonoid of $H/\mathrm{nim}\times B_1$ which is given by the observation that if the nim sum of a number of piles is at least 2, then there has to be a ``large'' pile. We can denote the elements of the quotient by $0, 1, \hat{0}, \hat{1}, \hat{2},\dots$, where a hat means that there is a large pile. These elements are added by adding the numbers by nim-addition, and putting a hat on the sum if at least one of the terms has a hat. They are mapped to outcomes by mapping $1$ and $\hat{0}$ to $\mathcal{P}$, and the remaining elements to $\mathcal{N}$.  

\subsection{Moore's nim}
Another variation of nim which is mentioned in several places in the literature is Moore's nim \cite{Moore}. This game too is played with piles of counters, so it is natural to again impose the additive structure of the heap monoid $H$. A move consists in removing any number of counters from at most $k-1$ piles, where $k\geq 2$ is a fixed integer. The player unable to make a move loses. The case $k=2$ is ordinary nim. In the standard formulation, a player has to remove counters from at least one and at most $k-1$ piles. In another form of the game, one is allowed to add counters to some of the piles, as long as the total number of counters decreases, and at most $k-1$ piles are affected. The two games are equivalent in the sense that the $\mathcal{P}$-positions are exactly the same. We let $$\mathrm{Moore}_k:H\to \{\mathcal{N}, \mathcal{P}\}$$ denote the outcome function. This function is computed in a rather peculiar way. First, the number of counters in the piles are written in binary. Then these numbers are added ``modulo $k$ without carry''. If this sum is zero, the position is a $\mathcal{P}$-position, otherwise it is an $\mathcal{N}$-position. 

Again we see that the description of the set of $\mathcal{P}$-positions goes through a monoid homomorphism. The quotient $H/\mathrm{Moore}_k$ is isomorphic to a direct sum of an infinite number of copies of $\mathbf{Z}_k$. 

From our point of view this shows that the algebraic method can work even in cases where a move can change several components. Already if $k = 3$, it is impossible to split the game into two components if it is required that each move should affect only one of the components.  

In principle Moore's nim can be regarded as a classical impartial game, but in the classical setting it does not decompose into components. See also the paper \cite{ES} by Richard Ehrenborg and Einar Steingr\`{\i}msson for a generalization of Moore's nim. 

\subsection{Nim with a comply-constrain twist}
In \cite{SS}, Furman Smith and Pantelimon St$\breve{\text{a}}$nic$\breve{\text{a}}$ introduced a class of games called comply-constrain games. These are games where a player, after making a move, puts a constraint on the opponent's next move. This sort of twist to a game is also called a \emph{Muller twist} after Muller, the inventor of the game \emph{Quarto}. The game of Odd-or-Even-nim, solved in \cite{SS}, is generalized two higher moduli by Hillevi Gavel and Pontus Strimling in \cite{GS}. 

For simplicity, we discuss a variant of comply-constrain nim without referring to modular arithmetic. It follows from the analysis that this game is essentially equivalent to the game called \emph{$k$-blocking modular nim} in \cite{GS}. See also \cite{AN2} for other nim-like games with constraints.

This game too is played on the heap monoid $H$. Before each move, there is a \emph{constraint-flag}, placing a restriction on the available moves. The constraint-flag is a set of $k-1$ positive integers which are forbidden move-sizes. The player to move has to remove a positive number of counters from a single pile, but this number cannot belong to the constraint set. After making a move, the player puts a new constraint by naming a new set of $k-1$ positive integers. A player unable to move loses, and this will clearly happen when all pile sizes are smaller than $k$. The player who first obtains such a position will win by putting the constraint $\{1,\dots,k-1\}$, making it impossible for the opponent to make another move. 

The presence of the constraint-flag makes the classification of positions into outcome classes a little more complicated. An outcome should specify what the result under optimal play should be (Previous or Next player win), for every possible constraint. Since there is, in principle, an infinite number of different constraints, there is potentially an uncountable infinity of different outcome classes. However, it turns out that only finitely many of these actually occur.  

The solution of the game, which is worked out in detail in \cite{GS}, can be described as follows: For each pile, the size is written as $q_ik + r_i$, where $q_i$ and $r_i$ are the quotient and remainder with division of the pile-size by $k$. The information we need in order to determine the outcome of a position is (i) the nim-sum of the $q_i$'s, and (ii) the maximum of the $r_i$'s. A position is a Previous-player win if (i) the nim sum of the $q_i$'s is zero, and (ii) the constraint forbids any number up to and including the maximum of the $r_i$'s. 

Again the information (i) and (ii) is ``homomorphic'' in the sense that if we know it for two positions, we know it for their sum. The quotient of $H$ with respect to this game is isomorphic to the cartesian product of the nim group with a linear order of size $k$. A linear order is a lattice, and in general, when we speak of a lattice as a monoid, we refer to the fact that a lattice is an abelian monoid under the operation of taking least upper bound. 

A fact that emerges from the solution in \cite{GS} is that whatever the position is, an optimal constraint is to forbid the numbers $\{1,\dots,k-1\}$.    

\subsection{Mis\`ere impartial games}
We have already mentioned the well-known solution of mis\`ere nim. In 1992, Mis\`ere Kayles was solved by W. Sibert and J. Conway \cite{SC}. In the more recent paper \cite{P05}, T. Plambeck developed a general theory for mis\`ere impartial games. It was shown that the quotient monoid of mis\`ere Kayles is finite, containing 48 elements. The theory was developed with the mis\`ere octal game 0.123 as an example. This game is completely solved by an explicit description of its 20-element monoid structure. In \cite{PS08, P}, several other mis\`ere games are solved, with more or less explicit descriptions of their monoid structures. 

We do not describe this theory in more detail here, but encourage the reader to take a look at the lecture notes \cite{Siegel} and its references.

\subsection{Hackenbush with a natural refinement} The game of hackenbush was introduced in \cite{WW}. It is a two-person \emph{partizan} game, i. e. the move options from a position are different for the two players. The special case that we consider here is called \emph{hackenbush strings} in the literature, but for brevity we refer to it as just hackenbush. The game is played with counters stacked in piles. The counters come in two colors, red and blue. The two players Red and Blue move alternately, and a move consists in removing a counter of the player's own color, together with all counters that are stacked above it. A player unable to move loses.

As with the games of nim-type, we split a position into components by regarding each pile as an irreducible component. We let $H_2$ denote the \emph{two-colored heap monoid}, that is, the free abelian monoid over the set of piles of red and blue counters. It follows from the analysis in \cite{WW} that the quotient $H_2/\text{hackenbush}$ is isomorphic to the set of rational numbers with a power of 2 in the denominator. This set will be denoted $\mathbf{Z}_{(2)}$ since algebraically it is the localization of the integers to the prime ideal $(2)$. 

In the terminology of \cite{WW, ONAG}, hackenbush positions are \emph{Numbers}. Without going deeper into the theory of Numbers, we describe how to calculate the outcome of a position in hackenbush by the so-called \emph{Berlekamp's rule}. We associate with each counter a \emph{weight}. If a counter has no counter of the opposite color anywhere below it, it has weight 1. Otherwise, it has half of the weight of the counter immediately below it. In other words, the counters below the first change of color will all have weight 1, while the counters above the first color change will have weights $1/2, 1/4, 1/8$, etc. Traditionally our sympathies are with Blue, so we now sum the weights of all the counters, with a minus-sign for the red ones. If the sum is positive, Blue wins, if it is negative, Red wins, while if it is zero, the player not to move wins.

Berlekamp's rule defines a monoid homomorphism of $H_2$ to $\mathbf{Z}_{(2)}$. Interestingly, each number of this form has a unique representation as a single-pile position. In perfect analogy with nim, each position is equivalent to a uniquely determined single-pile position. 

The proof of the correctness of Berlekamp's rule is fairly straightforward and carries over to some quite natural refinements of the game. In the following, we refer to the number associated to a position by Berlekamp's rule as its \emph{value}. The following theorem summarizes the necessary ingredients of the analysis of the game.

\begin{Thm} \label{T:hackenbush} If a pile is completely blue, then its value is an integer, Red has no move option, but Blue can decrease the value of the pile by 1. Conversely, if the pile is completely red, then Blue has no move option, but Red can increase the value by 1. If a pile is bichromatic, then its value can be written $a/2^k$, where $a$ is an odd integer, and $k\geq 1$. Blue has an option that decreases the value by $1/2^k$, and Red has an option that increases the value by $1/2^k$.
\end{Thm}

In each case, the move options in question consist in removing the topmost counter of one's own color. By Theorem \ref{T:hackenbush}, it follows that (i) if the value of a position is nonnegative, then every move by Red will make the value of the new position positive, and (ii) if the value of a position is positive, then Blue has a move option that makes the value of the new position nonnegative. The characterization of positive, negative and zero positions follows.

The value of a position in hackenbush somehow measures to what extent it favors one of the players, and the unit of measurement is moves. A monochromatic single pile of $n$ counters is worth $n$ moves for the player of that color. Therefore the following variant of the game is quite natural.

In \emph{integral hackenbush}, the positions and move options are the same as in ordinary hackenbush, but the game stops as soon as it reaches a monochromatic position. When this occurs, a player who has counters left receives one dollar for each of them from the other player. It is also possible that the game terminates at the empty position, in case no payment is made. This game differs from all the two-person games mentioned so far in that there are more than two potential results of the game. We think of the result as an integer describing the payment from Blue's point of view. An outcome is a pair of integers $(m,n)$, where $m$ is the result with Blue to move and $n$ is the result with Red to move.

\begin{Thm} The outcome of integral hackenbush is obtained by rounding the value of the position to the nearest integer. If the value is half an odd integer, then the tiebreak rule is that the value is rounded in favor of the player who is not to move.
\end{Thm}

This too is a simple consequence of Theorem \ref{T:hackenbush} and we omit the details of its proof. Notice that neither of the two games of hackenbush and integral hackenbush is a refinement of the other, but they still give rise to the same additive structure when passing to the quotient. 

We can construct a simultaneous refinement of the two games. Suppose that we play integral hackenbush strings, but that in addition to trying to remain with as many counters as possible in the end, the players also have a slight preference for making the last move: Beside the payment for remaining counters in the terminal position, a bonus of a quarter is to be paid to the player who makes the last move (that is, the last move before the position becomes monochromatic). We call this game \emph{refined integral hackenbush}. Optimal play in this game must be optimal for both ordinary and integral hackenbush strings. The following is another consequence of Theorem \ref{T:hackenbush}. 

\begin{Thm}
The outcome of refined integral hackenbush is obtained by rounding the value of the position to the nearest rational number with a smallest denominator of 4.
\end{Thm}

In the theory of \emph{thermography} \cite{ONAG}, one considers playing a game under the condition that the game terminates as soon as one reaches a position which is a Number. This number is then regarded as the result of the game. Similarly, one could play an ``integral'' version of any partizan game by terminating the game as soon as an integer is reached. 

\section{Trick-taking games} \label{S:trick-taking} \emph{Trick-taking games} is a family of card games of which surprisingly little can be found in the literature on combinatorial games. In the early days of contract bridge, the mathematician and former chess world champion Emanuel Lasker investigated a two-player model of the game, involving only one suit \cite{Lasker}. This and related games have subsequently been studied in \cite{KLW1, KLW2, KLW3, JW1, JW2}, but this seems to be an essentially complete bibliography. 

There are two players called \emph{West} and \emph{East}. The game is played with a set (deck) of \emph{cards}. Each card belongs to a \emph{suit}, and the cards within a suit are totally ordered by \emph{rank}. We do not limit the number of suits or the number of cards in a suit, and we do not assume that there are equally many cards in each suit. 

A game position is called a \emph{deal}, and consists of a distribution of cards to the players, that is, each player receives a set of cards called their \emph{hand}, and the East and West hands are disjoint. We say that a deal is \emph{balanced} if East and West have equally many cards. Many real-world trick-taking games require the deal to be balaced. Moreover, a deal is \emph{symmetric} if in each suit, the players have the same number of cards. 

One of the players is said to have the \emph{lead} (or \emph{be on lead}). The player on lead plays a card from their hand. The other player has to \emph{follow suit} if possible, that is, to play a card of the same suit as the card being led, and otherwise discard another card. These two cards constitute a \emph{trick}. The winner of the trick is the player who played the highest card in the suit that was led. The cards in the trick are removed, and the player who won the trick gets the lead. The game proceeds until all the cards have been played.

\subsection{Splitting into single-suit deals} We impose a monoid structure on the set of deals by considering each deal as a formal sum of its single-suit components. This is possible since all the suits have the same status. For example, we need not distinguish between the deals

\begin{equation}
\begin{array}{l l}
West:  & East:\\
\spadesuit \,\text{10 8 5} \qquad & \spadesuit \,\text{ J 9}\\
\heartsuit \,\text{ K } \qquad & \heartsuit \, \text{ Q 8}\\
\end{array}
\end{equation}

and 

\begin{equation}
\begin{array}{l l}
West:  & East:\\
\heartsuit \,\text{ 10 8 5} \qquad & \heartsuit \, \text{ J 9}\\
\diamondsuit \,\text{ K} \qquad & \diamondsuit \,\text{ Q 8}\\
\end{array}
\end{equation}

We think of both these deals as the sum $\left[10, 8, 5| \mathrm{J}, 9\right] + \left[\mathrm{K}| \mathrm{Q}, 8\right]$. As long as the rules of the games do not distinguish between the suits, these deals will have the same outcome. We let $\mathbf{D}$ be the monoid of all deals. The classes of balanced and symmetric deals are submonoids of $\mathbf{D}$, but unfortunately the set of balanced deals is not closed under splitting into single-suit components. Since most trick-taking games require the deals to be balanced, it may be necessary to restrict our attention to the symmetric case.

\subsection{Five-card} \emph{Five-card} is a traditional Swedish card game in which the objective is to win the last trick. We consider the two-player perfect information form of this game. The single-suit case was solved in \cite{JW1}. In principle, this solution carries over to the multi-suit case, although multi-suit games were not considered in that paper. Five-card is played with the so called \emph{greedy rule}, which means that playing second in a trick, one has to win the trick if possible.

It turns out to be sufficient to consider symmetric positions, but this is not obvious from the rules of the game, so to begin with, we just assume that the deal is symmetric. We define the \emph{trace} of a deal as follows. In each suit, compare the highest card on West's hand with the highest card on East's hand, then the second card on West's hand with the second card on East's hand, etc. For each comparison, count $+1$ if West's card is higher, and $-1$ if East's card is higher. The sum is the trace of the deal. It was shown in \cite{JW1} that in the single-suit case, West wins if the trace is positive, East wins if the trace is negative, and the player on lead wins if the trace is zero. 

The reason for this is that (1) If the player on lead wins all the comparisons in the computation of the trace, then they can maintain this by leading the smallest card in some suit. (2) Otherwise the player on lead can lead a card that loses its comparison in the computation of the trace, and make sure that the trace increases. (3) The player not on lead can play so that either the trace increases (if the opponent leads a high card) or they obtain the lead and the trace decreases by 1. From these three statements, the solution follows by induction.

Once the solution of the symmetric case is established, it is relatively easy to generalize it to general balanced deals. The \emph{symmetrization} of a deal is obtained by deleting cards with low rank in the suits where a player has more cards than the opponent, so that the deal becomes symmetric.

\begin{Thm}
The outcome of five-card under optimal play is the same as the outcome of the symmetrization of the deal. 
\end{Thm} 

It turns out that there is no point in leading a card where the opponent is void if one has another card (but clearly this holds only under perfect information and not in the real-world card game). 

\subsection{Symmetric mis\`ere five-card}
Suppose instead that the player who wins the last trick loses. This completely changes the strategy. The single-suit game is trivial, since the player who has the smallest card, the ``deuce'', wins by saving it for the last trick. With several suits, it is no longer true that the general case reduces to the symmetric case. Here we only discuss the symmetric form of the game, although a solution to the general case has been found by Bj\"orn Thal\'en. In the symmetric form it turns out that the greedy rule is superfluous, since a player will always want to win the trick unless there is no choice. 

The most important feature of a suit is the location of the deuce. A player who has the deuce in some suit is said to have an \emph{exit} in that suit. Exits don't add, you just have one or you don't. The quality of an exit is determined by the number of \emph{stoppers} that the player with the deuce has to protect it. The number of stoppers in a suit is the maximum number of tricks that the player can conceivably take in that suit, given that they have to save the deuce for the last trick. More precisely, remove the lowest card from each of the two hands, and pair up the remaining cards so that the player who had the deuce wins the maximum number of comparisons. That number is the number of stoppers protecting the exit.

For example consider the following deal:
\begin{equation}
\begin{array}{l l}
West:  & East:\\
\heartsuit \,\text{ K 2} \qquad & \heartsuit \, \text{ A Q}\\
\diamondsuit \,\text{ A J 10} \qquad & \diamondsuit \,\text{ K Q 2}\\
\clubsuit \,\text{ A Q 2} \qquad & \clubsuit \,\text{ K J 10}\\
\end{array}
\end{equation}

In hearts West has an exit, but if East leads the ace, West will have to play the king, and herefore the exit is unprotected (no stopper). In diamonds East has an exit protected by one stopper, and in clubs, West has an exit protected by two stoppers. 

Stoppers add like integers, in other words the relevant number is your total number of stoppers minus your opponent's total number of stoppers.

The outcome of the game is determined by the exits and the stoppers. The stoppers form a group isomorphic to $\mathbf{Z}$, while the exits form a lattice $B_2$ of four elements encoding the information for each of East and West whether they have an exit or not. Thus the outcome of the game can be found from the mapping of deals to $\mathbf{Z}\oplus B_2$. The player with more stoppers will win. If stoppers add to zero, the player with an exit will win over a player without exit, while if both players have exits, the player who starts will win. If nobody has an exit, the deal must be empty and the player who starts must already have won the last trick and lost the game.

Just as for mis\`ere nim there are superfluous elements in this monoid, since one cannot have a stopper without having an exit. There is also superfluous information, since if one player has a protected stopper, any unprotected stoppers for the other player become irrelevant. Therefore the quotient monoid is obtained by disregarding exits if the stoppers add to a nonzero number. More precisely, let $E$ be the quotient of $\mathbf{Z}\oplus B_2$ obtained by letting $(m,x)\equiv (m,y)$ whenever $m\neq 0$. Then symmetric mis\`ere five-card is isomorphic to $E$.
 
\subsection{Symmetric whist}
In \emph{whist}, the objective is to take as many tricks as possible. This game was studied in the single-suit case in \cite{JW1} and under the symmetry assumption in \cite{JW2}. See also \cite{W05} where the theory is applied to some unusual endgames in bridge. We do not state the results in detail here, but some of its features are worth commenting on. In \cite{JW2} the game was solved in the sense that its algebraic structure (as defined in this paper) was determined. Remarkably, this structure is the product of the group $\mathbf{Z}_{(2)}$ with the monoid $E$ of the previous section. In \cite{JW2} the elements of $E$ are called \emph{infinitesimals} because of their role in whist: The number of tricks that a player gets with optimal play is determined by rounding the value in $\mathbf{Z}_{(2)}$ to the nearest integer, with the tiebreak rule exactly the same as the rule for deciding the outcome in symmetric mis\`ere five-card. The rule is thus similar to the one for integral hackenbush, and the similarity even goes further: By introducing a bonus of a quarter for not taking the last trick, the outcome is obtained by rounding to nearest number with denominator 4.

Another fact that emerges from the analysis in \cite{JW2} is that the algebraic structure of whist is independent of whether the game is played with the greedy rule or not, although this rule definitely changes the outome of many deals. Moreover, the solution in \cite{JW2} is strictly speaking not effective in the computational sense, and several questions about the set of values that occur for single-suit deals are left unanswered. It seems that the overall character of the game determines its algebraic structure in a way which is independent of questions about the values of particular deals, and which is robust even under changes of the rules that perturb the actual mapping of deals to values. 

\subsection{General two-person whist}
Two-person whist without the assumption of symmetry seems to fall outside the scope of the current theory. The fact that one may force the opponent to discard by leading a suit where the opponent is void (in bridge terminology, execute a \emph{squeeze}) makes the analysis considerably more complicated. The following conjecture is worth mentioning:

\begin{Conjecture}
A higher card is always at least as good as a smaller card in the same suit. 
\end{Conjecture}

More precisely, if a card on say West's hand is removed and replaced by a higher card in the same suit, then West will be able to take at least as many tricks as in the original position. There is a strategy-stealing argument \cite{KLW1} that works for one suit but not in general, and it turns out that the conjecture is equivalent to the statement that having the lead can cost at most one trick compared to not having the lead in the same position.

\section{Comparison of the monoid theory with the classical theory}
At a superficial level, the main difference between the general algebaic approach taken in this paper and on the other hand the classical theory of \cite{WW, ONAG} is that the equivalence classes of game states form a monoid rather than a group. 

At a deeper level, we can say that while the classical theory focuses on how to add games and how to play several games at once, the approach taken here is to split game positions into components. In the classical theory, games are identified with their starting positions, and any two games can be added. Here we do not identify a game with its starting position. Instead, we regard a game as specifying a set of positions, and those positions can be added to each other only within the limits of the game. In our approach, there is no meaning in adding positions from different games. 

A fundamental difference between the theories, as noted in \cite{P05} and \cite{JW2}, is that in the classical theory the game equivalence classes and the algebraic identities between them are \emph{intrinsic}. To decide whether or not two games $G$ and $H$ are equivalent (in \cite{ONAG}, the relation that we here call \emph{equivalence} is called as \emph{equality}), all we have to do is play the sum $G+H$. If this sum is a second-player win, $G$ and $H$ are equivalent. The fact that this holds if and only if for every game $K$, $G+K$ belongs to the same outcome class as $H+K$ is then a nontrivial theorem.

Here, as well as in \cite{P05, PS08, P, Propp, JW2}, we have taken the latter property as our definition of equivalence. As was discussed in \cite{P05} and \cite{JW2}, this means that to decide whether two positions $G$ and $H$ are equivalent, we have to consider all possibilities for $K$. Hence even if the positions $G$ and $H$ are described by finite combinatorial structures, it is not obvious whether there is a decision procedure for determining whether or not equivalence holds. Moreover, if we extend the game to a larger set of positions, then two positions that are equivalent in the smaller game can turn out not to be equivalent in the larger game.  

\section{Non-commutative game theory} In \cite{WW, ONAG}, a theory is developed in which the games are elements of a group. In this paper we have shown that by relaxing the conditions, allowing for games that correspond to abelian monoids, we can incorporate a number of other games into the theory. One may ask why we should stop at abelian monoids. Are there alternative game theories where games correspond to other algebraic structures? In \cite{AN1}, a game called \emph{End-nim} is solved. This game, which comes in two forms, impartial and partizan, is played on a set of piles of counters, but the piles are ordered from left to right, and the ordering is essential to the game. It would be possible, and quite natural, to consider the set of positions as a non-commutative monoid. It is certainly possible to express the solution given in \cite{AN1} in this language. This example in itself does not motivate it, but it would be possible to develop a ``non-commutative game theory'' for counter pick-up games played on ordered sequences of piles.

\end{document}